\documentclass[10pt,twoside]{article}
\usepackage{mathrsfs}
\usepackage{amsmath}
\usepackage{amssymb}
\usepackage{fancyhdr}
\usepackage{latexsym}
\usepackage{bbding}
\usepackage{mathrsfs}
\usepackage{wasysym}

\usepackage{multicol,graphics}

\newtheorem{theorem}{Theorem}[section]
\newtheorem{lemma}[theorem]{Lemma}
\newtheorem{corollary}[theorem]{Corollary}

\numberwithin{equation}{section}
\newenvironment{proof}[1][Proof]{\noindent\textbf{#1.} }{\hfill $\Box$}
\allowdisplaybreaks

 \makeatletter
\begin{document}
\title{{A Beale-Kato-Majda Blow-up Criterion for a Hydrodynamic System Modeling Vesicle and Fluid Interactions}\footnote{This work is partially supported by the National Natural Science Foundation of
China (11171357).}}
\author{Jihong Zhao\footnote{E-mail: zhaojihong2007@yahoo.com.cn.}\\
[0.2cm] {\small  Institute of Applied Mathematics, College of
Science, Northwest A\&F
University,}\\
{\small Yangling, Shaanxi 712100,  People's Republic of China}}
\date{}
\maketitle

\begin{abstract}
In this paper, we establish an analog of the Beale-Kato-Majda type
criterion for singularities of smooth solutions of a hydrodynamic
system modeling vesicle and fluid interactions. The result shows
that the maximum norm of the vorticity alone controls the breakdown
of smooth solutions.


\end{abstract}

\section{Introduction}

During the past several decades, there have been many experimental
and mathematical investigations focusing on  the formation and
dynamics of elastic vesicle membranes
\cite{ALV02,C02,DBV97,KS05,MSWD94,OH89,S93,SL95}. The single
component vesicles are elastic membranes containing a liquid and
surrounded by another liquid, which are possibly the simplest models
for the biological cells and molecules. Such vesicles can be formed
by certain amphiphilic molecules assembled in water to build
bilayers, and exhibit a rich set of geometric structures in various
mechanical, physical and biological environment \cite{DLW04,M05}.
Their equilibrium shapes can be characterized by minimizing the
following bending elastic energy of the membranes \cite{H73}:
\begin{equation}\label{eq1.1}
  E=\int_{\Gamma}\frac{k}{2}(H-c_{0})^{2}dS,
\end{equation}
where $\Gamma$ is the surface of vesicle membrane,
$H=\frac{k_{1}+k_{2}}{2}$ is the mean curvature of the membrane
surface with $k_{1}$ and $k_{2}$ as the principle curvatures,
$c_{0}$ is the spontaneous curvature which arises due to
inhomogeneities in the bilayer lipid membrane structure, and $k$ is
the bending modulus of the vesicle membrane.


In \cite{DLL07,DLW04}, by using the phase field approach, the
authors introduced the phase field Navier-Stokes vesicle fluid
interaction model for the vesicle shape dynamics and conducted
numerical simulations of the vesicle membrane deformation in flow
fields (see \cite{DLRW051,DLRW09,WD08} for further studies). In this
model, the vesicle membrane is described by a phase function $\phi$,
which is a labeling function defined on computational domain $Q$.
The function $\phi$ takes value $+1$ inside of the vesicle membrane
and $-1$ outside, with a thin transition layer of width
characterized by a small (compared to the vesicle size) positive
parameter $\varepsilon$. Obviously, the vesicle membrane $\Gamma$
coincides with the zero level set $\{x: \phi(x)=0\}$. The
convergence of the phase field model to the original sharp interface
model as the transition width of the diffuse interface
$\varepsilon\rightarrow0$ has been carried out in \cite{DLRW05}. On
the other hand, the viscous fluid is modeled by the incompressible
Navier-Stokes equations with unit density and with an external force
defined in terms of $\phi$.

As in \cite{DLL07}, for simplicity, we assume that $k$ is a positive
constant and $H_{0}=0$. The elastic bending energy \eqref{eq1.1}
will be approximated by a modified Willmore energy (cf.
\cite{DLW04})
\begin{equation}\label{eq1.2}
  E_{\varepsilon}(\phi)=\frac{k}{2\varepsilon}\int_{Q}|f(\phi)|^{2}dx
  \ \ \text{with}\ \
  f(\phi)=-\varepsilon\Delta\phi+\frac{1}{\varepsilon}(\phi^{2}-1)\phi,
\end{equation}
which depends on the interface transitional thickness $\varepsilon$.
Moreover, in order to keep the total volume and the surface area of
the vesicle membrane are conserved in time, two constraint
functionals for the vesicle volume and surface area are prescribed
by (cf. \cite{DLW04})
\begin{equation}\label{eq1.3}
  A(\phi)=\int_{Q}\phi \;dx, \ \
  B(\phi)=\int_{Q}\Big(\frac{\varepsilon}{2}|\nabla\phi|^{2}+\frac{1}{4\varepsilon}(|\phi|^{2}-1)^{2}\Big)dx.
\end{equation}
To enforce these constraints, two penalty terms were added to the
elastic bending energy $E_{\varepsilon}(\phi)$, and the approximate
elastic bending energy is given by (cf. \cite{DLW05,DLW06})
\begin{equation}\label{eq1.4}
  E(\phi)=E_{\varepsilon}(\phi)+\frac{1}{2}M_{1}(A(\phi)-\alpha)^{2}+\frac{1}{2}M_{2}(B(\phi)-\beta)^{2},
\end{equation}
where $M_{1}$ and $M_{2}$ are two penalty constants,
$\alpha=A(\phi_{0})$ and $\beta=B(\phi_{0})$ are determined by the
initial value of the phase function $\phi_{0}$.

In this paper, we consider the three dimensional phase field
Navier-Stokes vesicle fluid interaction model subjecting to the
periodic boundary conditions (i.e., in torus $\mathbb{T}^{3}$),
which reads as follows:
\begin{align}\label{eq1.5}
 &{\partial_{t}}u+u\cdot\nabla u+\nabla{P}=\mu\Delta u+\frac{\delta E(\phi)}{\delta\phi}\nabla\phi\ \ &\text{in}\ \ Q\times[0,T],\\
\label{eq1.6}
 &\nabla\cdot u=0\ \ &\text{in}\ \ Q\times[0,T],\\
\label{eq1.7}
 &\partial_{t}\phi+u\cdot\nabla\phi=-\gamma\frac{\delta
 E(\phi)}{\delta\phi}\ \ &\text{in}\ \ Q\times[0,T]\;
\end{align}
with the initial condition
\begin{equation}\label{eq1.8}
  u(x,0)=u_{0}(x) \ \text{with}\ \nabla\cdot u_{0}=0\  \ \text{and}\ \ \phi(x,0)=\phi_{0}(x)\ \
  \text{for}\ \
  x\in Q,
\end{equation}
and the boundary condition
\begin{equation}\label{eq1.9}
  u(x+e_{i},t)=u(x,t), \quad \phi(x+e_{i},t)=\phi(x,t )\ \
  \text{for}\ \
  x\in\partial Q\times[0,T],
\end{equation}
where the set of vectors
$\{e_{1}=(1,0,0),e_{2}=(0,1,0),e_{3}=(0,0,1)\}$ denotes an
orthonormal basis of $\mathbb{R}^{3}$ and $Q$ is the unit square in
$\mathbb{R}^{3}$. Here $u=(u_{1},u_{2}, u_{3})\in\mathbb{R}^{3}$ and
$P\in\mathbb{R}$ denote, respectively, the velocity field and the
pressure of the fluid, $\phi\in\mathbb{R}$ is the phase function of
the vesicle membrane. $\frac{\delta E(\phi)}{\delta\phi}$ is the
so-called chemical potential that denotes the variational derivative
of $E(\phi)$ in the variable $\phi$.  $\gamma$ denotes the mobility
coefficient which is assumed to be a small positive constant. Note
that, if we denote
\begin{equation}\label{eq1.10}
  g(\phi)=-\Delta f(\phi)+\frac{1}{\varepsilon^{2}}(3\phi^{2}-1)f(\phi),
\end{equation}
then
\begin{align}\label{eq1.11}
  \frac{\delta
  E(\phi)}{\delta\phi}&=kg(\phi)+M_{1}(A(\phi)-\alpha)+M_{2}(B(\phi)-\beta)f(\phi)\nonumber\\
  &=k\varepsilon\Delta^{2}\phi-\frac{k}{\varepsilon}\Delta(\phi^{3}-\phi)-\frac{k}{\varepsilon}(3\phi^{2}-1)\Delta\phi+
  \frac{k}{\varepsilon^{3}}(3\phi^{2}-1)(\phi^{2}-1)\phi\nonumber\\
  &\ \ \ +M_{1}(A(\phi)-\alpha)+M_{2}(B(\phi)-\beta)f(\phi).
\end{align}

The system \eqref{eq1.5}--\eqref{eq1.7} describes the evolution of
vesicle membranes immersed in an incompressible viscous fluid.
Equations \eqref{eq1.5} and \eqref{eq1.6} are the momentum
conservation and the mass conservation equations of a viscous fluid
with unit density and with an external force caused by the phase
field $\phi$. Equation \eqref{eq1.6} is the condition of
incompressibility. Equation \eqref{eq1.7} is a relaxed transport
equation of $\phi$ with advection by the velocity field $u$. The
right-hand side of \eqref{eq1.7} is a regularization term which
ensures the consistent dissipation of energy. Roughly speaking, the
system \eqref{eq1.5}--\eqref{eq1.7} is governed by the coupling of
the hydrodynamic fluid flow and the bending elastic properties of
the vesicle membrane. The resulting membrane configuration and the
flow field reflect the competition and the coupling of the kinetic
energy and membrane elastic energies.

Based on the following basic energy law:
\begin{equation}\label{eq1.12}
  \frac{d}{dt}\Big(\frac{1}{2}\|u(\cdot,t)\|_{L^{2}}^{2}+E(\phi(\cdot,t))\Big)+\mu\|\nabla u(\cdot,t)\|_{L^{2}}^{2}
  +\gamma\|\frac{\delta E(\phi)}{\delta\phi}\|_{L^{2}}^{2}=0,\ \ \forall\
  t>0,
\end{equation}
the global existence of weak solution to the system
\eqref{eq1.5}--\eqref{eq1.7} subject to no-slip boundary condition
for the velocity field and Dirichlet boundary condition for the
phase function has been established in \cite{DLL07} via the Galerkin
argument. Moreover, the authors also proved the weak solution is
unique under an additionally regularity assumption $u\in L^{8}(0,T;
L^{4}(Q))$. Recently, local in time existence and uniqueness of
strong solution to the system \eqref{eq1.5}--\eqref{eq1.7} have been
established in \cite{LTT12}, and under the assumption that the
initial data and the quantity $(|\Omega|+\alpha)^{2}$ are
sufficiently small, the authors proved existence of almost global
strong solutions. Note that they have to restrict the working space
with proper limited regularity due to some compatibility conditions
at the boundary which is required in the fixed point strategy. Very
recently, Wu and Xu \cite{WX12} considered the system
\eqref{eq1.5}--\eqref{eq1.7} with periodic boundary conditions to
avoid troubles caused by the boundary terms when performing
integration by parts. They proved that, for any given initial data
$(u_{0}, \phi_{0})\in H^{1}_{per}(Q)\times H^{4}_{per}(Q)$, there
exists a positive time $T$ such that the system
\eqref{eq1.5}--\eqref{eq1.9} admits a unique smooth solution $(u,
\phi)$ satisfying
\begin{equation}\label{eq1.13}
\begin{cases}
  u\in C([0,T], H^{1}_{per}(Q))\cap L^{2}(0,T; H^{2}_{per}(Q))\cap H^{1}(0,T; L^{2}_{per}(Q)),\\
  \phi \in C([0,T], H^{4}_{per}(Q))\cap L^{2}(0,T; H^{6}_{per}(Q))\cap
  H^{1}(0,T; H^{2}_{per}(Q)).
\end{cases}
\end{equation}
Moreover, if the viscosity $\mu$ is assumed to be properly large,
then the system \eqref{eq1.5}--\eqref{eq1.9} admits a unique global
strong solution that is uniformly bounded in $H^{1}_{per}\times
H^{4}_{per}$ on $[0,\infty)$. However, as for the well-known
Navier-Stokes equations, an outstanding open problem is whether or
not smooth solution of \eqref{eq1.5}--\eqref{eq1.9} on $[0,T)$ will
lead to a singularity at $t=T$.  For this purpose, they established
some regularity criteria for local smooth solutions in terms of the
velocity field only. More precisely, they proved that if one of the
following conditions holds:
\begin{align}\label{eq1.14}
  &(i)\ \ \int_{0}^{T}\|u(\cdot, t)\|_{L^{r}}^{q}\;dt<\infty\
  \ \  \text{with}\ \
  \frac{3}{r}+\frac{2}{q}\leq 1,\ 3< r\leq \infty,\\
\label{eq1.15}
  &(ii) \ \ \int_{0}^{T}\|\nabla u(\cdot,t)\|_{L^{r}}^{q}\;dt<\infty\
  \ \  \text{with}\ \
  \frac{3}{r}+\frac{2}{q}\leq 2,\ \frac{3}{2}< r\leq \infty,
\end{align} then the solution $(u,\phi)$ can be extended beyond the time $T$.
Some improved logarithmical-type regularity criteria as for the
conventional Navier-Stokes equations were also established in
\cite{WX12}, and the stability of the system
\eqref{eq1.5}--\eqref{eq1.9} near local minimizers of the elastic
bending energy were obtained by using the {\L}ojasiewicz-Simon
inequality.

For the Navier-Stokes equations, the regularity criteria
\eqref{eq1.14} and \eqref{eq1.15} were obtained by Prodi \cite{P59},
Serrin \cite{S62}, Giga \cite{G86} and Beir\~{a}o da Veiga
\cite{B95}. In order to replace $\nabla u$ by the vorticity
$\nabla\times u$ in the condition \eqref{eq1.15}, the case
$1<r<\infty$ is a simple case because the Biot-Savart law and the
boundness of the Riesz transforms on $L^{r}$. However, the marginal
case $r=\infty$ is very difficult due to the lack of continuity of
Riesz transforms on $L^{\infty}$. In 1984, Beale-Kato-Majda in their
pioneering work \cite{BKM84} got around this difficulty by using the
logarithmic Sobolev inequality and showed that if the smooth
solution $u$ blows up at the time $t=T$, then
\begin{equation}\label{eq1.16}
  \int_{0}^{T}\|\omega(\cdot,
  t)\|_{L^{\infty}}\;dt=\infty,
\end{equation}
where $\omega=\nabla\times u$ is the vorticity of the velocity
field. Later, Kozono-Taniuchi \cite{KT00} and Konozo-Ogawa-Taniuchi
\cite{KOT02} refined \eqref{eq1.16} to
\begin{equation}\label{eq1.17}
  \int_{0}^{T}\|\omega(\cdot, t)\|_{BMO}\;dt=\infty\
  \ \text{and}\ \ \int_{0}^{T}\|\omega(\cdot, t)\|_{\dot{B}^{0}_{\infty,
  \infty}}\;dt=\infty,
\end{equation}
respectively, where $BMO$ is the space of  \textit{Bounded Mean
Oscillation} and $\dot{B}^{0}_{\infty, \infty}$ is the homogeneous
Besov spaces.

Inspired by \cite{BKM84} and \cite{WX12}, the purpose of this paper
is to establish an analog of Beale-Kato-Majda's criterion for
singularities of smooth solutions to the system
\eqref{eq1.5}--\eqref{eq1.9}. Our result shows that the maximum norm
of the vorticity alone controls the breakdown of smooth solutions,
and more important, it reveals that the velocity field $u$ plays a
more dominant role than the phase function $\phi$ in the regularity
theory of solutions to the phase field Navier-Stokes vesicle-fluid
interaction system \eqref{eq1.5}--\eqref{eq1.9}. As in \cite{BKM84},
the proof will be accomplished in three steps: obtaining $L^{2}$
estimates for the vorticity $\omega$ and $\nabla\Delta\phi$,
obtaining higher energy estimates for the solution $(u,\phi)$, and
applying the crucial logarithmic Sobolev inequality.

Before stating our main result, we recall some well-established
functional settings for periodic problems: For $1\leq r\leq \infty$,
we denote by
\begin{equation*}
  L^{r}_{per}(Q):=\{u\in L^{r}(\mathbb{R}^{3})\ |\  u(x+e_{i})=u(x)\}
\end{equation*}  equipped the usual norm
\begin{equation*}
 {\|u\|_{L^{r}}=}
\begin{cases}
  \big(\int_{Q}|u(x)|^{r}dx\big)^{1/r}\ \ &\mbox{if}\;\; 1\leq r<\infty,\\
  \text{ess}\sup_{x\in Q} |u(x)|\ \ &\mbox{if}\;\; r=\infty.
\end{cases}
\end{equation*}
 For an interger $m>0$, we denote by
\begin{equation*}
  H^{m}_{per}(Q):=\{u\in H^{m}(\mathbb{R}^{3})\ |\ u(x+e_{i})=u(x)\}
\end{equation*} equipped with the usual norm
$\|u\|_{H^{m}}=\sum_{i=0}^{m}\|\nabla^{i}u\|_{L^{2}}$.

\medskip

Our main result reads as follows:

\begin{theorem}\label{th1.1}
Assume that $(u_{0}, \phi_{0})\in H^{3}_{per}(Q)\times
H^{6}_{per}(Q)$ with $\nabla\cdot u_{0}=0$. Let  $T_{*}>0$ be the
maximum existence time such that the system
\eqref{eq1.5}--\eqref{eq1.9} has a unique strong solution $(u,\phi)$
on $[0,T_{*})$. If $T_{*}<\infty$, then
\begin{align}\label{eq1.18}
\int_{0}^{T_{*}}\|\omega(\cdot,t)\|_{L^{\infty}}\;dt=\infty.
\end{align}
In particular,
\begin{align*}
  \limsup_{t\nearrow T_{*}}\|\omega(\cdot,t)\|_{L^{\infty}}=\infty.
\end{align*}
\end{theorem}

The proof of Theorem \ref{th1.1} yields the following immediate
consequence:

\begin{corollary}\label{co1.2}
Assume that $(u_{0}, \phi_{0})\in H^{3}_{per}(Q)\times
H^{6}_{per}(Q)$ with $\nabla\cdot u_{0}=0$. Let $(u,\phi)$ be the
corresponding local strong solution to the system
\eqref{eq1.5}--\eqref{eq1.9} on [0,T) for some $0<T<\infty$. If
\begin{equation*}
\int_{0}^{T}\|\omega(\cdot,t)\|_{L^{\infty}}\;dt<\infty,
\end{equation*}
then the solution $(u,\phi)$ can be extended past time $t=T$.
\end{corollary}

In the next section, we give the proof of Theorem \ref{th1.1}.

\section{The proof of Theorem \ref{th1.1}}

We prove Theorem \ref{th1.1} by contradiction. Since $(u_{0},
\phi_{0})\in H^{3}_{per}(Q)\times H^{6}_{per}(Q)$ with $\nabla\cdot
u_{0}=0$, we deduce from \cite{WX12} that there exists $0<T<\infty$
such that the problem \eqref{eq1.5}--\eqref{eq1.9} admits a unique
strong solution $(u, \phi)$ satisfying
\begin{equation}\label{eq2.1}
\begin{cases}
  u\in C([0,T], H^{3}_{per}(Q))\cap L^{2}(0,T; H^{4}_{per}(Q))\cap H^{1}(0,T; H^{2}_{per}(Q)),\\
  \phi \in C([0,T], H^{6}_{per}(Q))\cap L^{2}(0,T; H^{8}_{per}(Q))\cap
  H^{1}(0,T; H^{4}_{per}(Q)).
\end{cases}
\end{equation}
Moreover, if $(u, \phi)$ does not belong to the class \eqref{eq2.1}
then it must be that
$$
  \limsup_{t\nearrow
  T}\big(\|u(\cdot,t)\|_{H^{3}}+\|\phi(\cdot,t)\|_{H^{6}}\big)=\infty.
$$
In fact, if this is not true, then there exists $C_{0}>0$ such that
$\|u(\cdot,t)\|_{H^{3}}+\|\phi(\cdot,t)\|_{H^{6}}\leq C_{0}$ for all
$t<T$. By the local existence theorem stated in \cite{WX12}, we can
start a solution at any time $t_{1}$ with initial value $(u(x,
t_{1}), \phi(x, t_{1}))$, and this solution will be regular for
$t_{1}<t<t_{1}+T_{0}(C_{0})$, with $T_{0}$ independent of $t_{1}$.
By choosing $t_{1}\in(T-T_{0}, T)$, we have then extended the
original solution past to the time $T$, contrary to the choice of T.

Therefore, it suffices to show that if the condition \eqref{eq1.18}
holds, then
\begin{equation}\label{eq2.2}
  \|u(\cdot,t)\|_{H^{3}}+\|\phi(\cdot,t)\|_{H^{6}}\leq C, \ \ 0\leq t\leq T
\end{equation}
for some constant $C$ depending only on $T$, $\|u_0\|_{H^{3}}$,
$\|\phi_0\|_{H^{6}}$, coefficients of the system
\eqref{eq1.5}--\eqref{eq1.9}, and
$$
  K:=\int_{0}^{T}\|\omega(\cdot,t)\|_{L^{\infty}}\;dt<\infty.
$$

Let us first establish $L^{2}$ estimate of the vorticity $\omega$
and $\nabla\Delta\phi$. Recall that we have already obtained the
following uniform estimates (cf. \cite{DLL07,WX12}):
\begin{equation}\label{eq2.3}
\begin{cases}
  \|u(\cdot,t)\|_{L^{2}}+\|\phi(\cdot, t)\|_{H^{2}}\leq C \ \ \text{for
  all}\ \
  t\geq0,\\
  \int_{0}^{+\infty}\big(\mu\|\nabla u(\cdot,t)\|_{L^{2}}^{2}+\gamma\|\frac{\delta E}{\delta \phi}(\cdot,t)\|_{L^{2}}^{2}\big)dt\leq
  C,
\end{cases}
\end{equation}
where $C$ is a constant depending only on $\|u_{0}\|_{L^{2}}$,
$\|\phi_{0}\|_{H^{2}}$ and coefficients of the system except the
viscosity $\mu$.

\begin{lemma}\label{le2.1}
Assume that $(u_{0}, \phi_{0})\in H^{3}_{per}(Q)\times
H^{6}_{per}(Q)$ with $\nabla\cdot u_{0}=0$. Let $(u,\phi)$ be the
corresponding local smooth solution to the system
\eqref{eq1.5}--\eqref{eq1.9} on [0,T) for some $0<T<\infty$. If
\begin{align}\label{eq2.4}
\int_{0}^{T}\|\omega(\cdot,t)\|_{L^{\infty}}\;dt=K<\infty,
\end{align}
then
\begin{align}\label{eq2.5}
  \sup_{0\leq t\leq T}\Big(\|\omega(\cdot,t)\|_{L^{2}}^{2}+\|\nabla\Delta \phi(\cdot,t)\|_{L^{2}}^{2}\Big)\leq C,
\end{align}
where $C$ is a constant depending only on $\|u_{0}\|_{H^{1}}$,
$\|\phi_{0}\|_{H^{3}}$, $K$ and coefficients of the system.
\end{lemma}
\begin{proof}
Taking the curl on \eqref{eq1.5}, we obtain
\begin{equation}\label{eq2.6}
  \partial_{t}\omega-\mu\Delta
  \omega+u\cdot\nabla\omega=\omega\cdot\nabla
  u+\nabla\times(\frac{\delta E}{\delta\phi}\nabla\phi).
\end{equation}
Multiplying \eqref{eq2.6} by $\omega$ and integrating over $Q$,
\begin{equation}\label{eq2.7}
  \frac{1}{2}\frac{d}{dt}\|\omega\|_{L^{2}}^{2}+\mu\|\nabla\omega\|_{L^{2}}^{2}=\int_{Q}w\cdot\nabla
  u\cdot\omega dx-\int_{Q}\frac{\delta
  E}{\delta\phi}\nabla\phi\cdot\nabla\times\omega dx,
\end{equation}
where we have used the fact $\int_{Q}u\cdot\nabla\omega\cdot\omega
dx=0$ due to $\nabla\cdot u=0$. Since the Riesz operators are
bounded in $L^{2}$ and $\nabla
u=(-\Delta)^{-1}\nabla(\nabla\times\omega)$, we have $\|\nabla
u\|_{L^{2}}\leq C\|\omega\|_{L^{2}}$. This implies that
\begin{equation}\label{eq2.8}
  \Big{|}\int_{Q}w\cdot\nabla u\cdot\omega dx\Big{|}\leq
  C\|\omega\|_{L^{\infty}}\|\nabla u\|_{L^{2}}\|\omega\|_{L^{2}}\leq
  C\|\omega\|_{L^{\infty}}\|\omega\|_{L^{2}}^{2}.
\end{equation}
Applying Young's inequality and \eqref{eq2.3}, we have
\begin{align}\label{eq2.9}
  \Big{|}\int_{Q}\frac{\delta
  E}{\delta\phi}\nabla\phi\cdot\nabla\times\omega dx\Big{|}&\leq
  \frac{\mu}{2}\|\nabla\omega\|_{L^{2}}^{2}+C\|\frac{\delta
  E}{\delta\phi}\nabla\phi\|_{L^{2}}^{2}\nonumber\\
  &\leq
  \frac{\mu}{2}\|\nabla\omega\|_{L^{2}}^{2}+C\|\frac{\delta
  E}{\delta\phi}\|_{L^{2}}^{2}\|\nabla\phi\|_{L^{\infty}}^{2}\nonumber\\
  &\leq
  \frac{\mu}{2}\|\nabla\omega\|_{L^{2}}^{2}+C\|\frac{\delta
  E}{\delta\phi}\|_{L^{2}}^{2}\Big(\|\nabla\Delta\phi\|_{L^{2}}^{2}+\|\nabla\phi\|_{L^{2}}^{2}\Big)\nonumber\\
  &\leq
  \frac{\mu}{2}\|\nabla\omega\|_{L^{2}}^{2}+C\|\frac{\delta
  E}{\delta\phi}\|_{L^{2}}^{2}\Big(\|\nabla\Delta\phi\|_{L^{2}}^{2}+1\Big),
\end{align}
where we have used the Sobolev embedding $H^{2}(Q)\hookrightarrow
L^{\infty}(Q)$ and $\|\nabla\phi\|_{L^{\infty}}\leq
C\big(\|\nabla\Delta\phi\|_{L^{2}}+\|\nabla\phi\|_{L^{2}}\big)$.
Taking \eqref{eq2.8} and \eqref{eq2.9} into \eqref{eq2.7}, we obtain
\begin{equation}\label{eq2.10}
  \frac{d}{dt}\|\omega\|_{L^{2}}^{2}+\mu\|\nabla\omega\|_{L^{2}}^{2}\leq
  C\Big(\|\omega\|_{L^{\infty}}+\|\frac{\delta
  E}{\delta\phi}\|_{L^{2}}^{2}\Big)\Big(\|\omega\|_{L^{2}}^{2}+\|\nabla\Delta\phi\|_{L^{2}}^{2}+1\Big).
\end{equation}
\medskip

Taking $\Delta$ on \eqref{eq1.7}, multiplying the resultant by
$-\Delta^{2}\phi$, and integrating over $Q$, we obtain
\begin{align}\label{eq2.11}
  \frac{1}{2}\frac{d}{dt}\|\nabla\Delta\phi\|_{L^{2}}^{2}=-\int_{Q}\nabla\cdot(u\cdot\nabla\phi)\cdot\nabla\Delta^{2}\phi dx
  -\gamma\int_{Q}\nabla\frac{\delta
  E}{\delta\phi}\cdot\nabla\Delta^{2}\phi dx:=I_{1}+I_{2}.
\end{align}
By \eqref{eq2.3}, $I_{1}$ can be directly estimated as follows:
\begin{align}\label{eq2.12}
  I_{1}&\leq
  \frac{k\gamma\varepsilon}{8}\|\nabla\Delta^{2}\phi\|_{L^{2}}^{2}+C\|\nabla u\cdot\nabla\phi\|_{L^{2}}^{2}
  +C\|u\cdot\nabla^{2}\phi\|_{L^{2}}^{2}\nonumber\\
  &\leq \frac{k\gamma\varepsilon}{8}\|\nabla\Delta^{2}\phi\|_{L^{2}}^{2}+C\|\nabla u\|_{L^{2}}^{2}\|\nabla\phi\|_{L^{\infty}}^{2}
  +C\|u\|_{L^{6}}^{2}\|\nabla^{2}\phi\|_{L^{3}}^{2}\nonumber\\
  &\leq \frac{k\gamma\varepsilon}{8}\|\nabla\Delta^{2}\phi\|_{L^{2}}^{2}+C\|\nabla
  u\|_{L^{2}}^{2}(\|\nabla\Delta\phi\|_{L^{2}}^{2}+1)
  +C\|\nabla u\|_{L^{2}}^{2}\|\nabla^{2}\phi\|_{L^{2}}\|\nabla\Delta\phi\|_{L^{2}}\nonumber\\
  &\leq \frac{k\gamma\varepsilon}{8}\|\nabla\Delta^{2}\phi\|_{L^{2}}^{2}+C\|\nabla
  u\|_{L^{2}}^{2}(\|\nabla\Delta\phi\|_{L^{2}}^{2}+1),
\end{align}
where we have used the Sobolev embedding $H^{1}(Q)\hookrightarrow
L^{6}(Q)$ and the interpolation inequality
$\|\nabla^{2}\phi\|_{L^{3}}^{2}\leq
C\|\nabla^{2}\phi\|_{L^{2}}\|\nabla\Delta\phi\|_{L^{2}}$. For
$I_{2}$, since $A(\phi)$ and $B(\phi)$ are functions depending only
on time, by \eqref{eq1.10} and \eqref{eq1.11}, we have
\begin{align}\label{eq2.13}
  I_{2}&=-\gamma\int_{Q}\nabla\Big[kg(\phi)+M_{1}(A(\phi)-\alpha)+M_{2}(B(\phi)-\beta)f(\phi)\Big]\cdot\nabla\Delta^{2}\phi
  dx\nonumber\\
  &=k\gamma\int_{Q}\nabla\Delta f(\phi)\cdot\nabla\Delta^{2}\phi
  dx-\frac{k\gamma}{\varepsilon^{2}}\int_{Q}\nabla[(3\phi^{2}-1)f(\phi)]\cdot\nabla\Delta^{2}\phi dx\nonumber\\
  &\ \ \
  -M_{2}\gamma(B(\phi)-\beta)\int_{Q}\nabla f(\phi)\cdot\nabla\Delta^{2}\phi
  dx\nonumber\\&:=I_{21}+I_{22}+I_{23}.
\end{align}
Note that
$f(\phi)=-\varepsilon\Delta\phi+\frac{1}{\varepsilon}(\phi^{2}-1)\phi$,
by  \eqref{eq2.3}, we obtain
\begin{align}\label{eq2.14}
  I_{21}&=-k\varepsilon\gamma\|\nabla\Delta^{2}\phi\|_{L^{2}}^{2}+\frac{k\gamma}{\varepsilon}\int_{Q}\nabla\Delta(\phi^{3}-\phi)\cdot\nabla\Delta^{2}\phi
  dx\nonumber\\
  &\leq
  -\frac{7k\varepsilon\gamma}{8}\|\nabla\Delta^{2}\phi\|_{L^{2}}^{2}+C\|\nabla\Delta(\phi^{3}-\phi)\|_{L^{2}}^{2}\nonumber\\
  &\leq
  -\frac{7k\varepsilon\gamma}{8}\|\nabla\Delta^{2}\phi\|_{L^{2}}^{2}+C\Big(\|\phi\|_{L^{\infty}}^{4}\|\nabla\Delta\phi\|_{L^{2}}^{2}
  +\|\phi\|_{L^{\infty}}^{2}\|\nabla\phi\|_{L^{6}}^{2}\|\Delta\phi\|_{L^{3}}^{2}\nonumber\\
  &\ \ \
  +\|\nabla\phi\|_{L^{6}}^{6}+\|\nabla\Delta\phi\|_{L^{2}}^{2}\Big)\nonumber\\
  &\leq
  -\frac{7k\varepsilon\gamma}{8}\|\nabla\Delta^{2}\phi\|_{L^{2}}^{2}+C\big(\|\nabla\Delta\phi\|_{L^{2}}^{2}+1\big);
\end{align}
\begin{align}\label{eq2.15}
  I_{22}&=-\frac{6k\gamma}{\varepsilon^{2}}\int_{Q}\phi\nabla\phi f(\phi)\cdot\nabla\Delta^{2}\phi dx
  -\frac{k\gamma}{\varepsilon^{2}}\int_{Q}(3\phi^{2}-1)\nabla f(\phi)\cdot\nabla\Delta^{2}\phi dx\nonumber\\
  &\leq
  \frac{k\varepsilon\gamma}{8}\|\nabla\Delta^{2}\phi\|_{L^{2}}^{2}+C\Big(\|\phi\nabla\phi f(\phi)\|_{L^{2}}^{2}
  +\|(3\phi^{2}-1)\nabla f(\phi)\|_{L^{2}}^{2}\Big)\nonumber\\
  &\leq
  \frac{k\varepsilon\gamma}{8}\|\nabla\Delta^{2}\phi\|_{L^{2}}^{2}+
  C\Big(\|\phi\nabla\phi\Delta\phi\|_{L^{2}}^{2}+\|\phi^{2}(\phi^{2}-1)\nabla\phi\|_{L^{2}}^{2}\nonumber\\
  &\ \ \ +\|(3\phi^{2}-1)\nabla\Delta\phi\|_{L^{2}}^{2}
  +\|(3\phi^{2}-1)\nabla(\phi^{3}-\phi)\|_{L^{2}}^{2}
  \Big)\nonumber\\
  &\leq
  \frac{k\varepsilon\gamma}{8}\|\nabla\Delta^{2}\phi\|_{L^{2}}^{2}+C\Big(\|\phi\|_{L^{\infty}}^{2}\|\nabla\phi\|_{L^{6}}^{2}\|\Delta\phi\|_{L^{3}}^{2}
  +\|\phi\|_{L^{\infty}}^{4}\|\phi^{2}-1\|_{L^{\infty}}^{2}\|\nabla\phi\|_{L^{2}}^{2}\nonumber\\
  &\ \ \ +
  \|3\phi^{2}-1\|_{L^{\infty}}^{2}\|\nabla\Delta\phi\|_{L^{2}}^{2}+\|3\phi^{2}-1\|_{L^{\infty}}^{4}\|\nabla\phi\|_{L^{\infty}}^{2}\Big)
  \nonumber\\
  &\leq
  \frac{k\varepsilon\gamma}{8}\|\nabla\Delta^{2}\phi\|_{L^{2}}^{2}+C\big(\|\nabla\Delta\phi\|_{L^{2}}^{2}+1\big);
\end{align}
\begin{align}\label{eq2.16}
  I_{23}&=M_{2}\varepsilon\gamma(B(\phi)-\beta)\int_{Q}\nabla\Delta\phi\cdot\nabla\Delta^{2}\phi
  dx-\frac{M_{2}\gamma(B(\phi)-\beta)}{\varepsilon}\int_{Q}\nabla(\phi^{3}-\phi)\cdot\nabla\Delta^{2}\phi
  dx\nonumber\\
  &\leq
  \frac{k\varepsilon\gamma}{8}\|\nabla\Delta^{2}\phi\|_{L^{2}}^{2}+C(B(\phi)-\beta)^{2}\Big( \|\nabla\Delta\phi\|_{L^{2}}^{2}
  +\|\nabla(\phi^{3}-\phi)\|_{L^{2}}^{2}\Big)\nonumber\\
  &\leq
  \frac{k\varepsilon\gamma}{8}\|\nabla\Delta^{2}\phi\|_{L^{2}}^{2}+C\Big(\|\nabla\phi\|_{L^{2}}^{4}
  +\|\phi^{2}-1\|_{L^{2}}^{4}+1\Big)\Big(\|\nabla\Delta\phi\|_{L^{2}}^{2}+\|\phi^{2}-1\|_{L^{\infty}}^{2}\|\nabla\phi\|_{L^{2}}^{2}\Big)
  \nonumber\\
  &\leq
  \frac{k\varepsilon\gamma}{8}\|\nabla\Delta^{2}\phi\|_{L^{2}}^{2}+C\big(\|\nabla\Delta\phi\|_{L^{2}}^{2}+1\big).
\end{align}
Combining the above estimates \eqref{eq2.11}--\eqref{eq2.16}, we
obtain
\begin{equation}\label{eq2.17}
  \frac{d}{dt}\|\nabla\Delta\phi\|_{L^{2}}^{2}+k\varepsilon\gamma\|\nabla\Delta^{2}\phi\|_{L^{2}}^{2}\leq
  C\big(\|\nabla u\|_{L^{2}}^{2}+1\big)\big(\|\nabla\Delta\phi\|_{L^{2}}^{2}+1\big).
\end{equation}

Putting \eqref{eq2.10} and \eqref{eq2.17} together yield that
\begin{align}\label{eq2.18}
  \frac{d}{dt}\big(\|\omega\|_{L^{2}}^{2}&+\|\nabla\Delta\phi\|_{L^{2}}^{2}\big)+\mu\|\nabla\omega\|_{L^{2}}^{2}
  +k\varepsilon\gamma\|\nabla\Delta^{2}\phi\|_{L^{2}}^{2}\nonumber\\
  &\leq
  C\Big(\|\omega\|_{L^{\infty}}+\|\nabla u\|_{L^{2}}^{2}+\|\frac{\delta
  E}{\delta\phi}\|_{L^{2}}^{2}+1\Big)\Big(\|\omega\|_{L^{2}}^{2}+\|\nabla\Delta\phi\|_{L^{2}}^{2}+1\Big).
\end{align}
Applying Gronwall's inequality, we conclude that
\begin{align}\label{eq2.19}
  \|\omega(t)\|_{L^{2}}^{2}+\|\nabla\Delta\phi(t)\|_{L^{2}}^{2}\leq
  C_{0}\exp\Big\{C\int_{0}^{t}\big(\|\omega(s)\|_{L^{\infty}}+\|\nabla u(s)\|_{L^{2}}^{2}+\|\frac{\delta
  E}{\delta\phi}(s)\|_{L^{2}}^{2}+1\big)ds\Big\},
\end{align}
where
$C_{0}=\big(\|\omega_{0}\|_{L^{2}}^{2}+\|\nabla\Delta\phi_{0}\|_{L^{2}}^{2}\big)$.
By \eqref{eq2.3} and \eqref{eq2.4}, we get \eqref{eq2.5}. The proof
of Lemma \ref{le2.1} is complete.
\end{proof}
\medskip

Combining \eqref{eq2.3} and \eqref{eq2.5}, we can easily see that
\begin{align}\label{eq2.20}
  \sup_{0\leq t\leq T}\|\phi(\cdot,t)\|_{H^{3}}\leq C,
\end{align}
where $C$ is a constant depending only on $\|u_{0}\|_{H^{1}}$,
$\|\phi_{0}\|_{H^{3}}$, $K$ and coefficients of the system.  By the
Sobolev embedding $H^{2}(Q)\hookrightarrow L^{\infty}(Q)$,
\eqref{eq2.20} implies that
\begin{align}\label{eq2.21}
  \sup_{0\leq t\leq T}\|\nabla\phi(\cdot,t)\|_{L^{\infty}}\leq C.
\end{align}
This result will be used frequently in the proof of Theorem
\ref{th1.1}.

\medskip

Next, let us derive higher order energy estimates of the solution
$(u,\phi)$. Taking $\nabla\Delta$ on \eqref{eq1.5}, multiplying the
resultant with $\nabla\Delta u$ and integrating over $Q$, we obtain
\begin{align}\label{eq2.22}
  \frac{1}{2}\frac{d}{dt}\|\nabla\Delta
  u\|_{L^{2}}^{2}+\mu\|\Delta^{2}u\|_{L^{2}}^{2}&=-\int_{Q}\nabla\Delta(u\cdot\nabla
  u)\cdot\nabla\Delta u dx+\int_{Q}\nabla\Delta(\frac{\delta E}{\delta
  \phi}\nabla\phi)\cdot\nabla\Delta u dx\nonumber\\
  &:=J_{1}+J_{2}.
\end{align}
Since $\nabla\cdot u=0$, $J_{1}$ can be rewritten as
\begin{equation*}
  J_{1}=-\int_{Q}\big[\nabla\Delta(u\cdot\nabla
  u)-u\cdot\nabla\nabla\Delta u\big]\cdot\nabla\Delta u dx.
\end{equation*}
By using the following commutator estimate due to Kato and Ponce
(see \cite{KP88}),
\begin{equation*}
   \|\nabla^{3}(fg)-f\nabla^{3}g\|_{L^{2}}\leq
   C\big(\|\nabla
   f\|_{L^{\infty}}\|\nabla^{2}g\|_{L^{2}}+\|\nabla^{3}f\|_{L^{2}}\|g\|_{L^{\infty}}\big),
\end{equation*}
we can estimate $J_{1}$ as follows:
\begin{align}\label{eq2.23}
  J_{1}&\leq C\|\nabla\Delta(u\cdot\nabla
  u)-u\cdot\nabla\nabla\Delta u\|_{L^{2}}\|\nabla\Delta
  u\|_{L^{2}}\nonumber\\
  &\leq C\|\nabla u\|_{L^{\infty}}\|\nabla\Delta
  u\|_{L^{2}}^{2}.
\end{align}
For $J_{2}$, after integration by parts, by \eqref{eq2.5}, we obtain
\begin{align}\label{eq2.24}
  J_{2}&=-\int_{Q}\Delta(\frac{\delta E}{\delta
  \phi}\nabla\phi)\Delta^{2}u dx\nonumber\\
  &\leq \frac{\mu}{4}\|\Delta^{2}u\|_{L^{2}}^{2}+C\|\Delta(\frac{\delta E}{\delta
  \phi}\nabla\phi)\|_{L^{2}}^{2}\nonumber\\
  &\leq \frac{\mu}{4}\|\Delta^{2}u\|_{L^{2}}^{2}+C\Big(\|\Delta\frac{\delta E}{\delta
  \phi}\nabla\phi\|_{L^{2}}^{2}+2\|\nabla\frac{\delta E}{\delta
  \phi}\nabla^{2}\phi\|_{L^{2}}^{2}+\|\frac{\delta E}{\delta
  \phi}\nabla\Delta\phi\|_{L^{2}}^{2}\Big)\nonumber\\
  &\leq \frac{\mu}{4}\|\Delta^{2}u\|_{L^{2}}^{2}+C\Big(\|\Delta\frac{\delta E}{\delta
  \phi}\|_{L^{2}}^{2}\|\nabla\phi\|_{L^{\infty}}^{2}+\|\nabla\frac{\delta E}{\delta
  \phi}\|_{L^{3}}^{2}\|\nabla^{2}\phi\|_{L^{6}}^{2}+\|\frac{\delta E}{\delta
  \phi}\|_{L^{\infty}}^{2}\|\nabla\Delta\phi\|_{L^{2}}^{2}\Big)\nonumber\\
  &\leq \frac{\mu}{4}\|\Delta^{2}u\|_{L^{2}}^{2}+C\Big(\|\Delta\frac{\delta E}{\delta
  \phi}\|_{L^{2}}^{2}+\|\frac{\delta E}{\delta
  \phi}\|_{L^{2}}^{2}+1\Big).
\end{align}
Combining \eqref{eq2.22}--\eqref{eq2.24}, we deduce that
\begin{align}\label{eq2.25}
  \frac{d}{dt}\|\nabla\Delta
  u\|_{L^{2}}^{2}+\frac{3\mu}{2}\|\Delta^{2}u\|_{L^{2}}^{2}\leq C\Big(\|\nabla u\|_{L^{\infty}}+\|\frac{\delta E}{\delta
  \phi}\|_{L^{2}}^{2}+1\Big)\Big(\|\nabla\Delta
  u\|_{L^{2}}^{2}+\|\Delta\frac{\delta E}{\delta
  \phi}\|_{L^{2}}^{2}+1\Big).
\end{align}
\medskip

To obtain the desired estimates for $\phi$, we start with
\begin{align}\label{eq2.26}
  \frac{1}{2}\frac{d}{dt}\|\Delta \frac{\delta
  E}{\delta\phi}\|_{L^{2}}^{2}&=\int_{Q}\frac{\partial}{\partial t}\Delta \frac{\delta
  E}{\delta\phi}\cdot\Delta \frac{\delta
  E}{\delta\phi}dx\nonumber\\
  &=\int_{Q}\frac{\partial}{\partial
  t}\Delta\big[kg(\phi)+M_{1}(A(\phi)-\alpha)+M_{2}(B(\phi)-\beta)f(\phi)\big]\cdot\Delta \frac{\delta
  E}{\delta\phi}dx\nonumber\\
  &=\int_{Q}\frac{\partial}{\partial
  t}\Delta\big[kg(\phi)+M_{2}(B(\phi)-\beta)f(\phi)\big]\cdot\Delta \frac{\delta
  E}{\delta\phi}dx\nonumber\\
  &=\int_{Q}\frac{\partial}{\partial
  t}\big[kg(\phi)+M_{2}(B(\phi)-\beta)f(\phi)\big]\cdot\Delta^{2} \frac{\delta
  E}{\delta\phi}dx\nonumber\\
  &=-k\int_{Q}\frac{\partial}{\partial
  t}\Delta f(\phi)\cdot\Delta^{2} \frac{\delta
  E}{\delta\phi}dx+\frac{k}{\varepsilon^{2}}\int_{Q}\frac{\partial}{\partial
  t}\big[(3\phi^{2}-1)f(\phi)\big]\cdot\Delta^{2} \frac{\delta
  E}{\delta\phi}dx\nonumber\\
  &\ \ \ +M_{2}\frac{d}{dt}B(\phi)\int_{Q}f(\phi)\cdot\Delta^{2}\frac{\delta
  E}{\delta\phi}dx+M_{2}(B(\phi)-\beta)\int_{Q}\frac{\partial}{\partial
  t}f(\phi)\cdot\Delta^{2}\frac{\delta
  E}{\delta\phi}dx\nonumber\\
  &:=K_{1}+K_{2}+K_{3}+K_{4}.
\end{align}
Now we are in a position to estimate the terms $K_{i}$ $(i=1,2,3,4)$
one by one. For $K_{1}$, we split it  into the following two parts:
\begin{align}\label{eq2.27}
  K_{1}&=k\varepsilon\int_{Q}\Delta^{2}\frac{\partial \phi}{\partial
  t}\cdot\Delta^{2} \frac{\delta
  E}{\delta\phi}dx-\frac{k}{\varepsilon}\int_{Q}\frac{\partial}{\partial
  t}\Delta(\phi^{3}-\phi)\cdot\Delta^{2} \frac{\delta
  E}{\delta\phi}dx\nonumber\\
  &:=K_{11}+K_{12}.
\end{align}
For $K_{11}$, by using Leibniz's rule, we deduce from \eqref{eq1.7}
that
\begin{align}\label{eq2.28}
  K_{11}&=-k\varepsilon\gamma\|\Delta^{2}\frac{\delta
  E}{\delta\phi}\|_{L^{2}}^{2}-k\varepsilon\int_{Q}\Delta^{2}(u\cdot\nabla\phi)\cdot\Delta^{2}\frac{\delta
  E}{\delta\phi}dx\nonumber\\
  &\leq -\frac{15k\varepsilon\gamma}{16}\|\Delta^{2}\frac{\delta
  E}{\delta\phi}\|_{L^{2}}^{2}+C\|\Delta^{2}(u\cdot\nabla\phi)\|_{L^{2}}^{2}\nonumber\\
  &\leq -\frac{15k\varepsilon\gamma}{16}\|\Delta^{2}\frac{\delta
  E}{\delta\phi}\|_{L^{2}}^{2}+C\Big(\|\Delta^{2}u\cdot\nabla\phi\|_{L^{2}}^{2}+4\|\nabla\Delta u\cdot\nabla\nabla\phi\|_{L^{2}}^{2}\nonumber\\
  &\ \ \ +6\|\Delta u\cdot\nabla\Delta\phi\|_{L^{2}}^{2}+4\|\nabla u\cdot\nabla\nabla\Delta\phi\|_{L^{2}}^{2}
  +\|u\cdot\nabla\Delta^{2}\phi\|_{L^{2}}^{2}\Big)\nonumber\\
  &\leq -\frac{15k\varepsilon\gamma}{16}\|\Delta^{2}\frac{\delta
  E}{\delta\phi}\|_{L^{2}}^{2}+C\Big(\|\nabla\phi\|_{L^{\infty}}^{2}\|\Delta^{2}u\|_{L^{2}}^{2}+\|\nabla\Delta
  u\|_{L^{3}}^{2}\|\nabla^{2}\phi\|_{L^{6}}^{2}\nonumber\\
  &\ \ \ +\|\Delta
  u\|_{L^{6}}^{2}\|\nabla\Delta\phi\|_{L^{3}}^{2}+\|\nabla
  u\|_{L^{\infty}}^{2}\|\Delta^{2}\phi\|_{L^{2}}^{2}+\|u\|_{L^{3}}^{2}\|\nabla^{5}\phi\|_{L^{6}}^{2}\Big)\nonumber\\
  &\leq -\frac{15k\varepsilon\gamma}{16}\|\Delta^{2}\frac{\delta
  E}{\delta\phi}\|_{L^{2}}^{2}+\tilde{C}\|\Delta^{2}u\|_{L^{2}}^{2}+C(\|\frac{\delta E}{\delta\phi}\|_{L^{2}}^{2}+1)(\|\nabla\Delta
  u\|_{L^{2}}^{2}+1)\nonumber\\
  &\ \ \ +C(\|\nabla u\|_{L^{2}}^{2}+1)(\|\Delta\frac{\delta
  E}{\delta\phi}\|_{L^{2}}^{2}+1)\nonumber\\
  &\leq -\frac{15k\varepsilon\gamma}{16}\|\Delta^{2}\frac{\delta
  E}{\delta\phi}\|_{L^{2}}^{2}+\tilde{C}\|\Delta^{2}u\|_{L^{2}}^{2}\nonumber\\
  &\ \ \ +C\Big(\|\nabla u\|_{L^{2}}^{2}+\|\frac{\delta E}{\delta\phi}\|_{L^{2}}^{2}+1\Big)
  \Big(\|\nabla\Delta
  u\|_{L^{2}}^{2}+\|\Delta\frac{\delta
  E}{\delta\phi}\|_{L^{2}}^{2}+1\Big),
\end{align}
where we have used the facts $\|\Delta^{2}\phi\|_{L^{2}}^{2}\leq
C(\|\frac{\delta E}{\delta\phi}\|_{L^{2}}^{2}+1)$ and
$\|\nabla^{5}\phi\|_{L^{6}}^{2}\leq
C\|\nabla^{6}\phi\|_{L^{2}}^{2}\leq C(\|\Delta\frac{\delta
E}{\delta\phi}\|_{L^{2}}^{2}+1)$. We emphasize here that the
constant $\tilde{C}$ in \eqref{eq2.28} depending only on
$\|u_{0}\|_{H^{1}}$, $\|\phi_{0}\|_{H^{3}}$, $K$ and coefficients of
the system due to the estimate \eqref{eq2.5}. For $K_{12}$, by using
\eqref{eq1.7} again, we obtain
\begin{align}\label{eq2.29}
  K_{12}&=-\frac{k}{\varepsilon}\int_{Q}\frac{\partial}{\partial
  t}\Delta(\phi^{3}-\phi)\cdot\Delta^{2} \frac{\delta
  E}{\delta\phi}dx
  =-\frac{6k}{\varepsilon}\int_{Q}\frac{\partial\big(|\nabla\phi|^{2}\phi\big)}{\partial
  t}\cdot\Delta^{2}\frac{\delta
  E}{\delta\phi}dx\nonumber\\&\ \ \ -\frac{3k}{\varepsilon}\int_{Q}\frac{\partial\big(\phi^{2}\Delta\phi\big)}{\partial
  t}\cdot\Delta^{2}\frac{\delta
  E}{\delta\phi}dx+\frac{k}{\varepsilon}\int_{Q}\Delta\frac{\partial\phi}{\partial
  t}\cdot\Delta^{2}\frac{\delta
  E}{\delta\phi}dx\nonumber\\
  &=-\frac{12k}{\varepsilon}\int_{Q}\phi\nabla\phi\nabla\frac{\partial \phi}{\partial t}\cdot\Delta^{2}\frac{\delta
  E}{\delta\phi}dx-\frac{6k}{\varepsilon}\int_{Q}|\nabla\phi|^{2}\frac{\partial\phi}{\partial
  t}\cdot\Delta^{2}\frac{\delta
  E}{\delta\phi}dx\nonumber\\
  &\ \ \ -\frac{6k}{\varepsilon}\int_{Q}\phi\Delta\phi\frac{\partial\phi}{\partial
  t}\cdot\Delta^{2}\frac{\delta
  E}{\delta\phi}dx-\frac{3k}{\varepsilon}\int_{Q}\phi^{2}\Delta\frac{\partial\phi}{\partial
  t}\cdot\Delta^{2}\frac{\delta
  E}{\delta\phi}dx\nonumber\\
  &\ \ \ +\frac{k}{\varepsilon}\int_{Q}\Delta\frac{\partial\phi}{\partial
  t}\cdot\Delta^{2}\frac{\delta
  E}{\delta\phi}dx:=\sum_{i=1}^{5}K_{12i}.
\end{align}
Note that, by using \eqref{eq2.5}, we can easily deduce from
\eqref{eq1.7} that
\begin{align*}
  \|\frac{\partial\phi}{\partial
  t}\|_{L^{2}}&\leq C\Big(\|u\cdot\nabla\phi\|_{L^{2}}+\|\frac{\delta
  E}{\delta\phi}\|_{L^{2}}\Big)\leq C\Big(\|\frac{\delta
  E}{\delta\phi}\|_{L^{2}}+1\Big),
\end{align*}
\begin{align*}
  \|\nabla\frac{\partial\phi}{\partial
  t}\|_{L^{2}}&\leq C\Big(\|\nabla u\cdot\nabla\phi\|_{L^{2}}+\|u\cdot\nabla^{2}\phi\|_{L^{2}}+\|\nabla\frac{\delta
  E}{\delta\phi}\|_{L^{2}}\Big)\\
  &\leq C\Big(\|\nabla u\|_{L^{2}}\|\nabla \phi\|_{L^{\infty}}+\|u\|_{L^{3}}\|\nabla^{2}\phi\|_{L^{6}}+\|\nabla\frac{\delta
  E}{\delta\phi}\|_{L^{2}}\Big)\\
  &\leq C\Big(\|\nabla u\|_{L^{2}}+\|\frac{\delta
  E}{\delta\phi}\|_{L^{2}}+\|\Delta\frac{\delta
  E}{\delta\phi}\|_{L^{2}}+1\Big),
\end{align*}
\begin{align*}
  \|\Delta\frac{\partial\phi}{\partial
  t}\|_{L^{2}}&\leq C\Big(\|\Delta u\cdot\nabla\phi\|_{L^{2}}+\|\nabla u\cdot\nabla^{2}\phi\|_{L^{2}}+\| u\cdot\nabla\Delta\phi\|_{L^{2}}
  +\|\Delta\frac{\delta
  E}{\delta\phi}\|_{L^{2}}\Big)\\
  &\leq C\Big(\|\Delta u\|_{L^{6}}\|\nabla \phi\|_{L^{3}}+\|\nabla u\|_{L^{3}}\|\nabla^{2}\phi\|_{L^{6}}
  +\|u\|_{L^{\infty}}\|\nabla\Delta\phi\|_{L^{2}}+\|\Delta\frac{\delta
  E}{\delta\phi}\|_{L^{2}}\Big)\\
  &\leq C\Big(\|\nabla\Delta u\|_{L^{2}}+\|\Delta\frac{\delta
  E}{\delta\phi}\|_{L^{2}}+1\Big),
\end{align*}
\begin{align*}
  \|\frac{\partial f(\phi)}{\partial
  t}\|_{L^{2}} &\leq C\Big(\|\Delta\frac{\partial\phi}{\partial t}\|_{L^{2}}+\|\frac{\partial}{\partial
  t}(\phi^{3}-\phi)\|_{L^{2}}\Big)\\
  &\leq C\Big(\|\Delta\frac{\partial\phi}{\partial t}\|_{L^{2}}+3\|\phi\|_{L^{\infty}}^{2}\|\frac{\partial\phi}{\partial
  t}\|_{L^{2}}+\|\frac{\partial\phi}{\partial
  t}\|_{L^{2}}\Big)\\
  &\leq C\Big(\|\nabla\Delta u\|_{L^{2}}
  +\|\Delta\frac{\delta
  E}{\delta\phi}\|_{L^{2}}+\|\frac{\delta
  E}{\delta\phi}\|_{L^{2}}+1\Big).
\end{align*}
Hence, we can estimate the terms $K_{12i}$ $(i=1,2,3,4,5)$ as
follows:
\begin{align}\label{eq2.30}
  K_{121}&\leq \frac{k\varepsilon\gamma}{16}\|\Delta^{2}\frac{\delta
  E}{\delta\phi}\|_{L^{2}}^{2}+C\|\phi\nabla\phi\nabla\frac{\partial\phi}{\partial
  t}\|_{L^{2}}^{2}\nonumber\\
  &\leq \frac{k\varepsilon\gamma}{16}\|\Delta^{2}\frac{\delta
  E}{\delta\phi}\|_{L^{2}}^{2}+C\|\phi\|_{L^{\infty}}^{2}\|\nabla\phi\|_{L^{\infty}}^{2}\|\nabla\frac{\partial\phi}{\partial
  t}\|_{L^{2}}^{2}\nonumber\\
  &\leq \frac{k\varepsilon\gamma}{16}\|\Delta^{2}\frac{\delta
  E}{\delta\phi}\|_{L^{2}}^{2}+C\Big(\|\nabla
  u\|_{L^{2}}^{2}
  +\|\Delta\frac{\delta E}{\delta\phi}\|_{L^{2}}^{2}+\|\frac{\delta
  E}{\delta\phi}\|_{L^{2}}^{2}\Big);
\end{align}
\begin{align}\label{eq2.31}
  K_{122}&\leq \frac{k\varepsilon\gamma}{16}\|\Delta^{2}\frac{\delta
  E}{\delta\phi}\|_{L^{2}}^{2}+C\|\nabla\phi\|_{L^{\infty}}^{4}\|\frac{\partial\phi}{\partial
  t}\|_{L^{2}}^{2}\nonumber\\
  &\leq \frac{k\varepsilon\gamma}{16}\|\Delta^{2}\frac{\delta
  E}{\delta\phi}\|_{L^{2}}^{2}+C\Big(\|\frac{\delta
  E}{\delta\phi}\|_{L^{2}}^{2}+1\Big);
\end{align}
\begin{align}\label{eq2.32}
  K_{123}&\leq \frac{k\varepsilon\gamma}{16}\|\Delta^{2}\frac{\delta
  E}{\delta\phi}\|_{L^{2}}^{2}+C\|\phi\|_{L^{\infty}}^{2}\|\frac{\partial\phi}{\partial
  t}\|_{L^{3}}^{2}\|\Delta\phi\|_{L^{6}}^{2}\nonumber\\
  &\leq \frac{k\varepsilon\gamma}{16}\|\Delta^{2}\frac{\delta
  E}{\delta\phi}\|_{L^{2}}^{2}+C\Big(\|u\cdot\nabla\phi\|_{L^{3}}^{2}+\|\frac{\delta
  E}{\delta\phi}\|_{L^{3}}^{2}\Big)\nonumber\\
  &\leq \frac{k\varepsilon\gamma}{16}\|\Delta^{2}\frac{\delta
  E}{\delta\phi}\|_{L^{2}}^{2}+C\Big(\|u\|_{L^{6}}^{2}\|\nabla\phi\|_{L^{6}}^{2}+\|\Delta\frac{\delta
  E}{\delta\phi}\|_{L^{2}}^{2}+\|\frac{\delta
  E}{\delta\phi}\|_{L^{2}}^{2}\Big)\nonumber\\
  &\leq \frac{k\varepsilon\gamma}{16}\|\Delta^{2}\frac{\delta
  E}{\delta\phi}\|_{L^{2}}^{2}+C\Big(\|\nabla u\|_{L^{2}}^{2}+\|\Delta\frac{\delta
  E}{\delta\phi}\|_{L^{2}}^{2}+\|\frac{\delta
  E}{\delta\phi}\|_{L^{2}}^{2}\Big);
\end{align}
\begin{align}\label{eq2.33}
  K_{124}&\leq \frac{k\varepsilon\gamma}{16}\|\Delta^{2}\frac{\delta
  E}{\delta\phi}\|_{L^{2}}^{2}+C\|\phi\|_{L^{\infty}}^{2}\|\Delta\frac{\partial\phi}{\partial
  t}\|_{L^{2}}^{2}\nonumber\\
  &\leq \frac{k\varepsilon\gamma}{16}\|\Delta^{2}\frac{\delta
  E}{\delta\phi}\|_{L^{2}}^{2}+C\Big(\|\nabla\Delta u\|_{L^{2}}^{2}+
  \|\Delta\frac{\delta
  E}{\delta\phi}\|_{L^{2}}^{2}+1\Big);
\end{align}
\begin{align}\label{eq2.34}
  K_{125}&\leq \frac{k\varepsilon\gamma}{16}\|\Delta^{2}\frac{\delta
  E}{\delta\phi}\|_{L^{2}}^{2}+C\|\Delta\frac{\partial\phi}{\partial
  t}\|_{L^{2}}^{2}\nonumber\\
  &\leq \frac{k\varepsilon\gamma}{16}\|\Delta^{2}\frac{\delta
  E}{\delta\phi}\|_{L^{2}}^{2}+C\Big(\|\nabla\Delta u\|_{L^{2}}^{2}+
  \|\Delta\frac{\delta
  E}{\delta\phi}\|_{L^{2}}^{2}+1\Big).
\end{align}
Putting estimates \eqref{eq2.30}--\eqref{eq2.34} together, we obtain
\begin{align}\label{eq2.35}
  K_{12}&\leq \frac{5k\varepsilon\gamma}{16}\|\Delta^{2}\frac{\delta
  E}{\delta\phi}\|_{L^{2}}^{2}+C\Big(\|\nabla\Delta u\|_{L^{2}}^{2}+
  \|\Delta\frac{\delta
  E}{\delta\phi}\|_{L^{2}}^{2}+\|\frac{\delta
  E}{\delta\phi}\|_{L^{2}}^{2}+1\Big).
\end{align}
Taking \eqref{eq2.28} and \eqref{eq2.35} into \eqref{eq2.27}, we get
\begin{align}\label{eq2.36}
  K_{1}&\leq -\frac{5k\varepsilon\gamma}{8}\|\Delta^{2}\frac{\delta
  E}{\delta\phi}\|_{L^{2}}^{2}+\tilde{C}\|\Delta^{2}u\|_{L^{2}}^{2}\nonumber\\
  &\ \ \ +C\Big(\|\nabla u\|_{L^{2}}^{2}+\|\frac{\delta E}{\delta\phi}\|_{L^{2}}^{2}+1\Big)
  \Big(\|\nabla\Delta u\|_{L^{2}}^{2}+
  \|\Delta\frac{\delta
  E}{\delta\phi}\|_{L^{2}}^{2}+1\Big).
\end{align}
For $K_{2}$,
\begin{align}\label{eq2.37}
  K_{2}&=\frac{6k}{\varepsilon^{2}}\int_{Q}\phi f(\phi)\frac{\partial\phi}{\partial
  t}\cdot\Delta^{2} \frac{\delta
  E}{\delta\phi}dx+\frac{k}{\varepsilon^{2}}\int_{Q}(3\phi^{2}-1)\frac{\partial f(\phi)}{\partial
  t}\cdot\Delta^{2} \frac{\delta
  E}{\delta\phi}dx\nonumber\\
  &=-\frac{6k}{\varepsilon}\int_{Q}\phi \Delta\phi\frac{\partial\phi}{\partial
  t}\cdot\Delta^{2} \frac{\delta
  E}{\delta\phi}dx+\frac{6k}{\varepsilon^{3}}\int_{Q}\phi^{2} (\phi^{2}-1)\frac{\partial\phi}{\partial
  t}\cdot\Delta^{2} \frac{\delta
  E}{\delta\phi}dx\nonumber\\
  &\ \ \ +\frac{k}{\varepsilon^{2}}\int_{Q}(3\phi^{2}-1)\frac{\partial f(\phi)}{\partial
  t}\cdot\Delta^{2} \frac{\delta
  E}{\delta\phi}dx\nonumber\\
  &\leq \frac{k\varepsilon\gamma}{8}\|\Delta^{2}\frac{\delta
  E}{\delta\phi}\|_{L^{2}}^{2}+C\Big(\|\phi\|_{L^{\infty}}^{2}\|\Delta\phi\|_{L^{6}}^{2}\|\frac{\partial\phi}{\partial
  t}\|_{L^{3}}^{2}
  +\|\phi\|_{L^{\infty}}^{4}\|\phi^{2}-1\|_{L^{\infty}}^{2}\|\frac{\partial\phi}{\partial t}\|_{L^{2}}^{2}\nonumber\\
  &\ \ \ +\|3\phi^{2}-1\|_{L^{\infty}}^{2}\|\frac{\partial f(\phi)}{\partial
  t}\|_{L^{2}}^{2}\Big)\nonumber\\
  &\leq \frac{k\varepsilon\gamma}{8}\|\Delta^{2}\frac{\delta
  E}{\delta\phi}\|_{L^{2}}^{2}+C\Big(\|\frac{\partial\phi}{\partial
  t}\|_{L^{3}}^{2}+\|\frac{\partial\phi}{\partial t}\|_{L^{2}}^{2}
  +\|\frac{\partial f(\phi)}{\partial
  t}\|_{L^{2}}^{2}\Big)\nonumber\\
  &\leq \frac{k\varepsilon\gamma}{8}\|\Delta^{2}\frac{\delta
  E}{\delta\phi}\|_{L^{2}}^{2}+C\Big(\|\nabla\Delta u\|_{L^{2}}^{2}+\|\Delta\frac{\delta E}{\delta \phi}\|_{L^{2}}^{2}
  +\|\frac{\delta E}{\delta \phi}\|_{L^{2}}^{2}+1\Big).
\end{align}
For $K_{3}$,
\begin{align}\label{eq2.38}
  K_{3}&\leq C\Big{|}\frac{d B(\phi)}{dt}\Big{|}\|f(\phi)\|_{L^{2}}\|\Delta^{2}\frac{\delta
  E}{\delta\phi}\|_{L^{2}}\nonumber\\
  &\leq \frac{k\varepsilon\gamma}{8}\|\Delta^{2}\frac{\delta
  E}{\delta\phi}\|_{L^{2}}^{2}+C\Big(\|\nabla\phi\|_{L^{2}}\|\nabla\frac{\partial \phi}{\partial t}\|_{L^{2}}
  +\|\phi^{3}-\phi\|_{L^{2}}\|\frac{\partial \phi}{\partial
  t}\|_{L^{2}}\Big)^{2}\|f(\phi)\|_{L^{2}}^{2}\nonumber\\
  &\leq \frac{k\varepsilon\gamma}{8}\|\Delta^{2}\frac{\delta
  E}{\delta\phi}\|_{L^{2}}^{2}+C\Big(\|\nabla u\|_{L^{2}}^{2}+\|\Delta\frac{\delta E}{\delta \phi}\|_{L^{2}}^{2}
  +\|\frac{\delta E}{\delta \phi}\|_{L^{2}}^{2}+1\Big).
\end{align}
For $K_{4}$,
\begin{align}\label{eq2.39}
  K_{4}&\leq C|B(\phi)-\beta|\|\frac{\partial f(\phi)}{\partial t}\|_{L^{2}}\|\Delta^{2}\frac{\delta
  E}{\delta\phi}\|_{L^{2}}\nonumber\\
  &\leq \frac{k\varepsilon\gamma}{8}\|\Delta^{2}\frac{\delta
  E}{\delta\phi}\|_{L^{2}}^{2}+C\Big(\|\nabla\phi\|_{L^{2}}^{2}
  +\|\phi^{2}-1\|_{L^{2}}^{2}\Big)^{2}\|\frac{\partial f(\phi)}{\partial t}\|_{L^{2}}^{2}\nonumber\\
  &\leq \frac{k\varepsilon\gamma}{8}\|\Delta^{2}\frac{\delta
  E}{\delta\phi}\|_{L^{2}}^{2}+C\Big(\|\nabla\Delta u\|_{L^{2}}^{2}+\|\Delta\frac{\delta E}{\delta \phi}\|_{L^{2}}^{2}
  +\|\frac{\delta E}{\delta \phi}\|_{L^{2}}^{2}+1\Big).
\end{align}
Taking \eqref{eq2.36}--\eqref{eq2.39} into \eqref{eq2.26}, we
conclude that
\begin{align}\label{eq2.40}
  \frac{d}{dt}\|\Delta \frac{\delta
  E}{\delta\phi}\|_{L^{2}}^{2}&+\frac{k\varepsilon\gamma}{2}\|\Delta^{2}\frac{\delta
  E}{\delta\phi}\|_{L^{2}}^{2}\leq \tilde{C}\|\Delta^{2}u\|_{L^{2}}^{2}\nonumber\\
  &\ \ \ +C\Big(\|\nabla u\|_{L^{2}}^{2}+\|\frac{\delta E}{\delta\phi}\|_{L^{2}}^{2}+1\Big)
  \Big(\|\nabla\Delta u\|_{L^{2}}^{2}+\|\Delta\frac{\delta
  E}{\delta\phi}\|_{L^{2}}^{2}+1\Big).
\end{align}
Set
$$
  \eta=\frac{\mu}{2\tilde{C}}.
$$
Multiplying $\eta$ by \eqref{eq2.40}, adding \eqref{eq2.25}
together, we obtain
\begin{align}\label{eq2.41}
  \frac{d}{dt}\Big(\|\nabla\Delta
  u\|_{L^{2}}^{2}&+\eta\|\Delta \frac{\delta
  E}{\delta\phi}\|_{L^{2}}^{2}\Big)+\mu\|\Delta^{2}u\|_{L^{2}}^{2}+\frac{k\varepsilon\gamma\eta}{2}\|\Delta^{2}\frac{\delta
  E}{\delta\phi}\|_{L^{2}}^{2}\nonumber\\
  &\leq C\Big(\|\nabla u\|_{L^{\infty}}+\|\nabla u\|_{L^{2}}^{2}+\|\frac{\delta E}{\delta\phi}\|_{L^{2}}^{2}+1\Big)
  \Big(\|\nabla \Delta u\|_{L^{2}}^{2}+\eta\|\Delta\frac{\delta
  E}{\delta\phi}\|_{L^{2}}^{2}+e\Big),
\end{align}
where the constant $C$ may depend on $\eta$. Putting the inequality
\eqref{eq2.41} and the basic energy inequality \eqref{eq1.12}
together imply that
\begin{align}\label{eq2.42}
  \frac{d}{dt}\Big(\|u\|_{H^{3}}^{2}\!+\!\eta\|\frac{\delta
  E}{\delta\phi}\|_{H^{2}}^{2}\Big)\!\leq\! C\Big(\|\nabla u\|_{L^{\infty}}\!+\!\|\nabla u\|_{L^{2}}^{2}
  \!+\!\|\frac{\delta E}{\delta\phi}\|_{L^{2}}^{2}\!+\!1\Big)
  \Big(\|u\|_{H^{3}}^{2}\!+\!\eta\|\frac{\delta
  E}{\delta\phi}\|_{H^{2}}^{2}\!+\!e\Big).
\end{align}
\medskip

Finally, we end the proof of Theorem \ref{th1.1} by applying
Gronwall's inequality. Set
$$
  m(t)=e+\|u\|_{H^{3}}^{2}+\eta\|\frac{\delta
  E}{\delta\phi}\|_{H^{2}}^{2}.
$$
Then by \eqref{eq2.42}, we see that
\begin{equation}\label{eq2.43}
  \frac{dm(t)}{dt}\leq C\Big(\|\nabla u\|_{L^{\infty}}+\|\nabla u\|_{L^{2}}^{2}+\|\frac{\delta E}{\delta\phi}\|_{L^{2}}^{2}+1\Big)m(t).
\end{equation}
By combining the following logarithmic
Sobolev inequality (see \cite{BKM84}):
$$
  \|\nabla u\|_{L^{\infty}}\leq
  C\big(1+\|\omega\|_{L^{2}}+\|\omega\|_{L^{\infty}}\ln(e+\|u\|_{H^{3}})\big)
$$
and \eqref{eq2.5}, we get
\begin{equation}\label{eq2.44}
  \|\nabla u\|_{L^{\infty}}\leq
  C\big(1+\|\omega\|_{L^{\infty}}\ln(e+\|u\|_{H^{3}})\big).
\end{equation}
This estimate with \eqref{eq2.43} and the inequality $\ln m(t)\geq
1$ yield that
\begin{align}\label{eq2.45}
  \frac{d}{dt}\ln m(t)\leq C\Big(\|\omega\|_{L^{\infty}}+\|\nabla u\|_{L^{2}}^{2}+\|\frac{\delta E}{\delta\phi}\|_{L^{2}}^{2}+1\Big)\ln m(t).
\end{align}
By applying Gronwall's inequality, we obtain
\begin{align}\label{eq2.46}
 m(t)\leq \exp\Big\{\ln m(0)\exp\big(C(1+t)+C\int_{0}^{t}\|\omega(s)\|_{L^{\infty}}ds\big)\Big\}
\end{align}
for any $0<t\leq T$, where $m(0)$ is a constant depending only on
$\|u_{0}\|_{H^{3}}$, $\|\phi_{0}\|_{H^{6}}$ and coefficients of the
system. Note that
\begin{equation*}
  \|\phi\|_{H^{6}}^{2}\leq C(\|\Delta\frac{\delta
  E}{\delta\phi}\|_{L^{2}}^{2}+1).
\end{equation*}
This fact with \eqref{eq2.46} imply immediately that
$$
  \sup_{0\leq t\leq T}\big(\|u(\cdot,t)\|_{H^{3}}+\|\phi(\cdot,t)\|_{H^{6}}\big)\leq
  C.
$$
We complete the proof of Theorem \ref{th1.1}.


\begin{thebibliography}{99}

\bibitem{ALV02} M. Abkarian, C. Lartigue, A. Viallat, Tank treading and unbinding of deformable
vesicles in shear flow: Determination of the lift force, Phys. Rev.
Lett. 88 (2002), 068103.

\bibitem{BKM84} J. T. Beale, T. Kato, A. Majda, Remarks on breakdown of smooth solutions for
the 3D Euler equations, Commun. Math. Phys. 94 (1984), 61--66.

\bibitem{B95} H. Beir\~{a}o da Veiga, A new regularity class for the Navier-Stokes equations in $\mathbb{R}^{n}$, Chin. Ann.
Math. Ser. B: 16 (1995), 407--412.

\bibitem{C02} R. S. Chadwick, Axisymmetric indentation of a thin incompressible elastic layer,
SIAM J. Appl. Math. 62(5) (2002), 1520--1530.

\bibitem{DBV97} K. H. de Haas, C. Blom, D. van den Ende, M. H. G. Duits, J.
Mellema, Deformation of giant lipid bilayer vesicles in shear flow,
Phys. Rev. E 56 (1997), 7132--7137.

\bibitem{DLL07} Q. Du, M. Li, C. Liu, Analysis of a phase field Navier-Stokes vesicle-fluid interaction
model, Discrete Contin. Dyn. Syst. B 8(3) (2007), 539--556.

\bibitem{DLRW05} Q. Du, C. Liu, R. Ryham, X. Wang, A phase field formulation of
the Willmore problem, Nonlinearity 18 (2005), 1249--1267.

\bibitem{DLRW051} Q. Du, C. Liu, R. Ryham, X. Wang, Phase field modeling of the
spontaneous curvature effect in cell membranes, Commun. Pure Appl.
Anal. 4 (2005), 537--548.

\bibitem{DLRW09} Q. Du, C. Liu, R. Ryham, X. Wang, Energetic variational approaches in modeling
vesicle and fluid interactions, Physica D 238 (2009), 923--930.

\bibitem{DLW04} Q. Du, C. Liu, X. Wang, A phase field approach in the numerical
study of the elastic bending energy for vesicle membranes, J.
Computational Physics 198 (2004), 450--468.

\bibitem{DLW05} Q. Du, C. Liu, X. Wang, Retrieving topological information for
phase field models, SIAM J. Appl. Math. 65 (2005), 1913--1932.

\bibitem{DLW06} Q. Du, C. Liu, X. Wang, Simulating the deformation of vesicle
membranes under elastic bending energy in three dimensions, J.
Computational Physics 212 (2006), 757--777.

\bibitem{H73} W. Helfrich, Elastic properties of lipid bilayers: theory and
possible experiments, Z. Naturforsch. 28 c (1973), 693--703.


\bibitem{G86} Y. Giga, Solutions for semilinear parabolic equations in $L^{p}$ and regularity of weak solutions of the
Navier-Stokes system, J. Differential Equations 61 (1986), 186--212.


\bibitem{KS05} V. Kantsler, V. Steinberg, Orientation and dynamics of a vesicle in tank-treading
motion in shear flow, Phys. Rev. Lett. 95 (2005), 258101.

\bibitem{KP88} T. Kato, G. Ponce, Commutator estimates and the Euler and Navier-Stokes equations,
Comm. Pure Appl. Math. 41 (1988), 891--907.

\bibitem{KOT02} H. Kozono, T. Ogawa, Y. Taniuchi, The critical
Sobolev inequalities in Besov spaces and regularity criterion to
some semi-linear evolution equations, Math. Z. 242 (2002), 251--278.

\bibitem{KT00} H. Kozono, Y. Taniuchi, Bilinear estimates in BMO and the Navier-Stokes equations,
Math. Z. 235 (2000), 173--194.



\bibitem{LTT12} Y. Liu, T. Takahashi, M. Tucsnak, Strong solution for a phase
field Navier-Stokes vesicle fluid interaction model, J. Math. Fluid
Mech. 14 (2011), 177--195.

\bibitem{MSWD94} L. Miao, U. Seifert, M. Wortis, H.-G D\"{o}bereiner, Budding transitions of fluid-bilayer vesicle:
The effect of area-difference elasticity, Phys. Rev. E 49 (1994),
5389--5407.

\bibitem{M05} O. Mouritsen, \textit{Life-As a Matter of Fat: The Emerging Science of
Lipidomics}, Springer, Berlin, 2005.

\bibitem{OH89} Z. Ou-Yang, W. Helfrich, Bending energy of vesicle membranes:
General expressions for the first, second and third variation of the
shape energy and applications to spheres and cylinders, Phys. Rev. A
39 (1989), 5280--5288.


\bibitem{P59} G. Prodi, Un teorema di unicit\`{a} per le equazioni di
Navier-Stokes, Ann. Math. Pura Appl. 48(1) (1959), 173--182.

\bibitem{S93} U. Seifert, Curvature-induced lateral phase segregation in two-component vesicles,
Phys. Rev. Lett. 70 (1993), 1335--1338.

\bibitem{SL95} U. Seifert, R. Lipowsky, \textit{Morphology of Vesicles}, in: Handbook of Biological Physics,
vol. 1, 1995.

\bibitem{S62} J. Serrin, On the interior regularity of weak solutions of the Navier-Stokes equations, Arch. Rational
Mech. Anal. 9 (1962), 187--195.

\bibitem{WD08} X. Wang, Q. Du, Modelling and simulations of multi-component
liquid membranes and open membranes via diffuse interface
approaches, J. Math. Biol. 56(3) (2008), 347--371.

\bibitem{WX12} H. Wu, X. Xu, Global regularity and stability of a hydrodynamic system
modeling vesicle and fluid interactions, arXiv:1202.4869v1.

\end{thebibliography}
\end{document}